\documentclass{amsart}
\usepackage{amssymb}
\usepackage{amsmath}
\usepackage{graphicx}
\usepackage{enumerate}
\usepackage{color}
\newcommand{\G}{\mathcal{G}}
\newcommand{\Z}{\mathbb{Z}}
\newcommand{\R}{\mathbb{R}}

\newcommand{\N}{\mathbb{N}}

\newcommand{\seq}[1]{\langle #1 \rangle}

\newtheorem{thm}{Theorem}[section]
\newtheorem{cor}[thm]{Corollary}
\newtheorem{rem}[thm]{Remark}
\newtheorem{que}[thm]{Question}
\newtheorem{lem}[thm]{Lemma}
\newtheorem{defn}[thm]{Definition}

\DeclareMathOperator{\dist}{dist}

\begin{document}
%

\title{Mixing completely scrambled system exists}
\author{Jan P. Boro\'nski}
\address[J. P. Boro\'nski]{AGH University of Science and Technology, Faculty of Applied
	Mathematics, al.
	Mickiewicza 30, 30-059 Krak\'ow, Poland
	-- and --
	National Supercomputing Centre IT4Innovations, Division of the University of Ostrava,
	Institute for Research and Applications of Fuzzy Modeling,
	30. dubna 22, 70103 Ostrava,
	Czech Republic}
\email{jan.boronski@osu.cz}
\author{Ji\v{r}\'{\i} Kupka}
\address[J. Kupka]{National Supercomputing Centre IT4Innovations, Division of the University of Ostrava,
	Institute for Research and Applications of Fuzzy Modeling,
	30. dubna 22, 70103 Ostrava,
	Czech Republic}
	\email{jiri.kupka@osu.cz}
\author{Piotr Oprocha}
\address[P. Oprocha]{AGH University of Science and Technology, Faculty of Applied
	Mathematics, al.
	Mickiewicza 30, 30-059 Krak\'ow, Poland
	-- and --
	National Supercomputing Centre IT4Innovations, Division of the University of Ostrava,
	Institute for Research and Applications of Fuzzy Modeling,
	30. dubna 22, 70103 Ostrava,
	Czech Republic}
\email{oprocha@agh.edu.pl}

\begin{abstract}
We prove that there exists a topologically mixing homeomorphism which is completely scrambled. We also prove that for any integer $n\geq 1$
there is a continuum of topological dimension $n$ supporting a transitive completely scrambled homeomorphism, and $n$-dimensional compactum supporting a weakly mixing completely scrambled homeomorphism. This solves a 15 year old open problem.
\end{abstract}
\maketitle

\section{Introduction}
The notion of Li-Yorke pair was introduced for the first time by Li and Yorke in their seminal paper  \cite{LY} from 1975.
Recall that $x,y\in X$ form a \textit{Li-Yorke pair} for a continuous map $T\colon X\to X$ on a compact metric space $(X,d)$ if
$$
\liminf_{n\to \infty} d(T^n(x),T^n(y))=0\quad \textrm{and} \quad \limsup_{n\to \infty} d(T^n(x),T^n(y))>0
$$
and a set $D\subset X$ is \textit{scrambled} when each pair of its distinct points is Li-Yorke. A natural question which arose immediatelly was about possible size of scrambled set.
While it had been known since the beginning that scrambled sets can be uncountable, it was soon after realized that when $X=[0,1]$ a scrambled set is never residual,
\cite{Gedeon} but it can have full Lebesgue measure \cite{Bruckner} (see also \cite{Mai}; for more detailed discussion on the size of scrambled sets see survey paper \cite{Snoha} and references therein).
Clearly the most extreme situation is when the whole $X$ is a scrambled set. In such a case we say that $(X,T)$ is \textit{completely scrambled}.
A question of whether a completely scrambled system can exist for a compact $X$ was open for a while, and a positive answer was brought by the paper \cite{HY}
by Huang and Ye.  They proved that completely scrambled sets exist on the Cantor set, and then used a nice extension method to show that for every $n\geq 1$ there exists
a compact set (in fact, a continuum) $X$ together with a completely scrambled homeomorphism $T\colon X\to X$.
Unfortunately their examples are not transitive, hence they were wondering if these examples can be improved to obtain a transitive map.
It has also been known that a completely scrambled transitive, or even completely scrambled and weakly mixing homeomorphism exist (see \cite{AAB} or its extension in \cite{FO})
for a closed subset $X$ in $[0,1]^\N$ of unknown dimension (most likely infinite). Therefore the following question has remained open since publication of \cite{HY} in 2001.
\begin{que}\label{que:1}
Does there exist for every $n\geq 1$ a continuum $X$ of dimension $n$ supporting a completely scrambled transitive homeomorphims?
\end{que}

In Section~\ref{sec:que} we will show that the answer to the above question is positive. We will also prove the following theorem which extends the results of \cite{FO}.
\begin{thm}\label{main}
There exists a mixing homeomorphism on the Cantor set which is completely scrambled.
\end{thm}
Note that the above result cannot be obtained by a direct modification of the approach in \cite{FO}, because uniformly rigid systems (a core property in the construction in \cite{FO}) are never mixing. In our approach we will use a technique of Shimomura \cite{Shim} which is based on the construction of an inverse limit of graph covers introduced by Akin, Glasner and Weiss in \cite{AGW}.
While in many aspects our construction in the proof of Theorem~\ref{main} is inspired by \cite{Shim}, our proofs are different and somehow shorter.

Using Theorem~\ref{main} and a class of examples discovered by Glasner and Maon \cite{Maon}, we prove the following:
\begin{thm}\label{main:2}
For every $n\geq 1$ there exists a compact set $Y$ of topological dimension $n$ supporting completely scrambled and weakly mixing homeomorphism.
\end{thm}

\section{Preliminaries}

A \textit{dynamical system} $(X,T)$ is a pair consisting of compact metric space $(X,d)$ and a continuous map $T\colon X\to X$.
By \textit{dimension}, we always mean topological dimension, and \textit{Cantor set} is any 0-dimensional compact metric space without isolated points, that is any space homeomorphic with the standard Cantor set in the real line. A \textit{continuum} is a compact connected metric space.

\subsection{Transitivity and chaos}

For each $x\in X$ we define its (positive) \textit{orbit} putting $O^+_T(x)=\{T^n(x): n\geq 0\}$ and its \textit{$\omega$-limit set}
by $\omega_T(x)=\bigcap_{n\geq 0}\overline{O^+_T(T^n(x))}$. A set $Z\subset X$ is \textit{invariant} if $T(Z)\subset Z$
and $Z$ is a subsystem if it is closed and invariant, which in other words means that $(Z,T)$ is a dynamical system.
A point $x\in X$ is \textit{recurrent} if $x\in \omega_T(x)$ and \textit{minimal} if it is recurrent and $\omega_T(x)=\omega_T(y)$
for every $y\in \omega_T(x)$. An invariant set $Z$ is \textit{minimal} if $Z=\omega_T(x)$ for some $x\in X$.

A pair of points $x,y\in X$ is:
\begin{enumerate}[(i)]
	\item  \textit{proximal} if $\liminf_{n\to \infty} d(T^n(x),T^n(y))=0$;
	\item \textit{asymptotic} if $\limsup_{n\to \infty} d(T^n(x),T^n(y))=0$;
\item \textit{Li-Yorke} if it is proximal but not asymptotic.
\end{enumerate}
A set $D\subset X$ is \textit{scrambled} if any two distinct $x,y\in D$ are a Li-Yorke pair. If there exists an uncountable scrambled set then we say that $(X,T)$ is \textit{chaotic} in the sense of Li and Yorke.
If the whole $X$ is a scrambled set then $(X,T)$ is \textit{completely scrambled}.

A dynamical system $(X,T)$ is:
\begin{enumerate}[(i)]
\item \textit{transitive} if for any nonempty open sets $U,V$ there is $n>0$ such that $T^n(U)\cap V \neq \emptyset$;
\item \textit{weakly mixing} if the product
system $(X\times X, T\times T)$ is transitive;
\item \textit{mixing} if for any two open sets $U,V$ there is $N$ such that $T^n(U)\cap V \neq \emptyset$ for all $n\geq N$.
\item \textit{uniformly rigid} if $\liminf_{n\to \infty}\rho(T^n,\textrm{id}_X)=0$, where $\rho$ is the metric of uniform convergence,
that is $\rho(S,R)=\sup_{x\in X} d(S(x),R(x))$.
\end{enumerate}

Let $T\colon X\to X$ be a homeomorphism and let $Z$ be a compact metric space obtained form $X\times [0,1]$ by identifying pairs
$(x,1)$ with $(T(x),0)$. The \textit{suspension flow over $(X,T)$} is the flow $\seq{\Phi_t}_{t\in \R}$ defined on $Z$ by the formula
$$
\Phi_t(s,x)=(\{s+t\}, T^{\lfloor s+t\rfloor}(x))
$$
where $\lfloor r\rfloor$ and $\{r\}$ denote the integer and fractional parts of a real number $r\in \R$, respectively.
Clearly, if $(X,T)$ is transitive then $Y$ is a continuum.

A continuous map $\pi\colon X \to Y$ is a \emph{factor} map between dynamical systems $(X,T)$ and $(Y,S)$ if $\pi$ is surjective
and $\pi \colon T = S\colon \pi$.

\subsection{Graph covers}
By a \textit{graph} we mean a pair $G=(V,E)$ of finite sets,
where $E\subset V\times V$. By a \textit{vertex} in $G$ we mean any element of $V$,
and elements of $E$ are called the \textit{edges}.
The graphs we consider are always \textit{edge surjective}, i.e.
for every $v\in V$ there are $u,w\in V$ such that $(u,v), (v,w)\in E$. In other words, each
vertex has an incoming and outgoing edge.
A map $\phi \colon V_1\to V_2$ is a \textit{homomorphism} between graphs $(V_1,E_1)$, $(V_2,E_2)$
if for every $(u,v)\in E_1$ we have $(\phi(u),\phi(v))\in E_2$; to emphasize that $\phi$ is a graph homomorphism
we write $\phi \colon (V_1,E_1)\to (V_2,E_2)$. For simplicity, for edge $e=(u,v)\in E_1$ we use notation $\phi(e)=(\phi(u),\phi(v))$. This extends onto paths
$e_1\ldots e_n$ on $(V_1,E_1)$ by the standard rule $\phi(e_1\ldots e_n)=\phi(e_1)\ldots \phi(e_n)$.
A homomorphism $\phi$ is \textit{bidirectional} if $(u,v),(u,v')\in E_1$ implies $\phi(v)=\phi(v')$ and
$(w,u),(w',u)\in E_1$ implies $\phi(w)=\phi(w')$.
If $\phi$ is a bidirectional map between edge-surjective graps then we call it \textit{bd-cover}. Let $\G=\seq{\phi_i}_{i=0}^\infty$ denote a sequence of bd-covers $\phi_i \colon (V_{i+1},E_{i+1})\to (V_i,E_i)$, and let
$$
V_\G=\varprojlim(V_i,\phi_i)=\{ x\in \Pi_{i=0}^\infty V_i : \phi_i(x_{i+1})=x_i \text{ for all }i\geq 0\}
$$
be the inverse limit defined by $\G$.
As usual, let $\phi_{m,n}=\phi_n\circ \phi_{n+1}\circ \ldots \circ \phi_{m-1}$ and denote the projection from $V_\G$ onto $V_n$ by $\phi_{\infty,n}$.
Set
$$
E_\G=\{(x,y)\in V_\G\times V_\G : (x_i,y_i)\in E_i \text{ for each }i=1,2,\dots \}
,$$
As usual, $V_i$ is endowed with discrete topology and $\mathbb{X}=\prod_{i=0}^\infty V_i$
is endowed with product topology. This topology is compatible with the metric given by $d(x,y)=0$ when $x=y$ and
$d(x,y)=2^{-k}$
when $x\neq y$ and $k=\min \{i : x_i\neq y_i\}$. Then $V_\G$ is a closed subset of $\mathbb{X}$ and we consider it with topology (and metric)
induced from $\mathbb{X}$.

By a cycle on graph $G$ we mean any path starting and ending in the same vertex.
If $c_1,\ldots, c_n$ are cycles starting in the same vertex $v$ then by $a_1 c_1+\ldots +a_n c_n$
we denote the cycle at $v$ obtained by passing $a_1$ times cycle $c_1$ then $a_2$ times cycle $c_2$, etc.
Length of a path $\eta$ (i.e. the number of edges on it) is denoted $|\eta|$. By $V(\eta)$ we denote the set of vertexes
on path $\eta$. The following important fact is given in \cite[Lemma 3.5]{Shim2}, with injectivity of $E_\G$ relying on the fact that the sequence is bidirectional.
\begin{lem}\label{lem:Tg}
Let $\G=\seq{\phi_i}$ be a sequence of bd-covers $\phi_i \colon (V_{i+1},E_{i+1})\to (V_i,E_i)$.
Then $V_\G$ is a zero-dimensional compact metric space and the relation $E_\G$ defines a homeomorphism.
\end{lem}

\section{A special homeomorphism}
\subsection{Construction of homeomorphism $(V_\G,T_\G)$}\label{sec:constr}

For every integer $n>0$ we define a special vertex $v_{n,0}\in V_n$ and a special edge $e_{n,0}=(v_{n,0},v_{n,0})\in E_n$ and we put $G_0=(\{v_{0,0}\},\{e_{0,0}\})$.
Next we will specify other vertexes in $V_n$ and accompanying edges, so that a graph $G_n$ is defined.
Our aim is to construct a special sequence of bd-covers. In particular we put $\phi_n(v_{n+1,0})=v_{n,0}$.
We embed in each $V_n$ exactly $n$ additional cycles $c_{n,1},\ldots, c_{n,n}$ (of appropriate length, which will be determined later), such that each cycle starts in $v_{n,0}$
and all the other vertexes are pairwise distinct. For $n=1$ we include in $V_1$ a cycle $c_{1,1}$ of length $10$ and define $\phi_1(c_{1,1})=10 e_{0,0}$.

We assume that length of each $c_{n,i}$ is appropriate so that the following bd-covers are well defined. For $2\leq i \leq n$ we put
$$
\phi_n(c_{n+1,i})=e_{n,0}+2c_{n,i}+\ldots+2c_{n,n}+e_{n,0}.
$$
We also put $\phi_n(c_{n+1,n+1})=(n+2)^2(\sum_{i=1}^n|c_{n,i}|)e_{n,0}$ and finally
\begin{eqnarray*}
\phi_n(c_{n+1,1})&=&e_{n,0} + 2 c_{n,1}+2e_{n,0}+2 c_{n,1}+\ldots\\&&\quad +k_n e_{n,0}+2 c_{n,1}
+ e_{n,0}+2c_{n,2}+\ldots+2c_{n,n}+e_{n,0},
\end{eqnarray*}
where $k_n=2(|e_{n,0}|+\sum_{i=1}^n |c_{n,i}|)$.
Since first and the last edge in each cycle $c_{n+1,i}$ is sent by $\phi_n$ onto $e_{n,0}$ each $\phi_n$ is a bd-cover.

\begin{defn}
Let $\G=\seq{\phi_i}_{i=0}^\infty$ be the sequence defined above and 
we denote by $T_\G\colon V_\G\to V_\G$ the homeomorphism induced by $E_\G$
in view of Lemma~\ref{lem:Tg}.
\end{defn}
\subsection{Properties of $(V_\G,T_\G)$}
In this section we prove a few auxiliary lemmas about properties of the system $(V_\G,T_\G)$ constructed in Section~\ref{sec:constr}. They will be used later in the proof of Theorem~\ref{main}.
For simplicity of notation we denote $X=V_\G$ and $T=T_\G$.

Denote $p=(v_{0,0},v_{1,0},v_{2,0},\ldots)$. It is not hard to see that $p\in V_\G$ and $(p,p)\in E_\G$. In particular, $p$ is a fixed point of $T$.
We easily obtain the following.

\begin{lem}\label{Lem:OnlyPeriodic}
The point $p$ is the only periodic point of $(X,T)$.
\end{lem}

\begin{proof}
Note that length of each cycle $c_{n,i}$ increases with $n$ and each such cycle has pairwise distinct vertexes.	
Hence if $z$ is a periodic point then for sufficiently large $n$ it may not contain any vertex from $c_{n,i}$, and so the only choice is $z_n=v_{n,0}$
for each $n$. Indeed $z=p$.
\end{proof}

\begin{lem}\label{lem:prox}
For every $x\in X$ we have $p\in \omega_T(x)$.
\end{lem}
\begin{proof}
Fix any integers $n>0$ and $m>0$ and let $y=T^n(x)$. If $T^n(x)_m=v_{m,0}$ then $T^n(x)_i=v_{i,0}$ for all $i\leq m$ and so $d(x,p)\leq 2^{-m}$.
In the case $T^n(x)_m\neq v_{m,0}$, vertex belongs $T^n(x)_m$ to some cycle $c_{m,j}$ and then in a finite number of iterations (say $l>0$) $T^{n+l}(x)_m$ reach terminal vertex of the cycle
$c_{m,j}$ which is $v_{m,0}$. This shows that for every $n$ there is $s\geq n$ such that $T^s(x)_m=v_{m,0}$ and so $d(T^s(x),p)\leq 2^{-m}$.
Therefore $d(\omega_{T}(x),p)\leq 2^{-m}$, and since $m$ can be arbitrarily large and $\omega_{T}(x)$ is closed, we have $p\in \omega_{T}(x)$.
\end{proof}

\begin{cor}\label{cor:prox}
Dynamical system $(X,T)$ is proximal.
\end{cor}
\begin{proof}
By Lemma~\ref{lem:prox} the fixed point $p$ belongs to every $\omega$-limit set, hence $\{p\}$ is the unique minimal subsystem in $X$.
But it is well known (e.g. see \cite{Akin}) that it equivalently means that $(X,T)$ is proximal.
\end{proof}	

\begin{lem}\label{lem:mixing}
Dynamical system $(X,T)$ is topologically mixing.
\end{lem}
\begin{proof}
By the definition of product topology, it is enough to show that for any $m$ and any vertexes $u,v\in V_m$ there is $N>0$ such that for every $n\geq N$
there is $x\in X$ such that $x_m=u$ and $T^n(x)_m=v$.
Observe that $\phi_m(c_{m+1,1})$ contains all vertexes from $V_m$. So it is enough to show that there is $N>0$ such that for every $n\geq 0$ there is
a point $x\in X$ such that path induced by $T^i(x)_{m+1}$ on $(V_{m+1},E_{m+1})$ contains $c_{m+1,1}+\eta+c_{m+1,1}$ as a subpath, with length of intermediate path $|\eta|=n$.
Note that each $c_{j,1}$ is an inner element of the path $\phi_j(c_{j+1,1})$, i.e. subpath not containing the first or last vertex, hence for any $j$ it is not hard to find $x\in V_\G$
such that $T^{i-1}(x)_j$ is exactly the $i$-th vertex of $c_{j,1}$ for each $0< i \leq |c_{j,1}|$.

We claim that:
\begin{enumerate}[(i)]
	\item\label{mix:c1} for each $j\geq 1$ the cycle $\phi_{m+j ,m}(c_{m+j,1})$ contains cycle $c_{m,1}+\eta_i+c_{m,1}$ with intermediate path $\eta_i$ of length $|\eta_i|=i$,
	where $i=0,\ldots, k_{m+j-1}$.
	\item\label{mix:c2} $j e_{m,0}+c_{m,1}$ is a prefix of $\phi_{m+j ,m}(c_{m+j,1})$,
	\item\label{mix:c3} there is $\alpha$ with $|\alpha|\leq k_{m+j-1}-j$ such that $c_{m,1}+\alpha$ is a suffix of $\phi_{m+j ,m}(c_{m+j,1})$.
\end{enumerate}
The claim is clear for $j=1$ since $c_{m,1}+k e_{m,0}+c_{m,1}$ is a subpath of $\phi_{m+1 ,m}(c_{m+1,1})=\phi_m(c_{m+1,1})$
for $k=0,\ldots,k_m$, $\phi_m(c_{m+1,1})$ starts with $e_{m,0}+ c_{m,1}$ and ends with path
$$c_{m,1}
+ e_{m,0}+2c_{m,2}+\ldots+2c_{m,m}+e_{m,0}$$ whose length is exactly $k_m-|c_{m,1}|\leq k_m-m$.

Next assume that the claim holds for some $j\geq 0$ and consider
$$
\xi=\phi_{m+j+1 ,m}(c_{m+j+1,1})=\phi_{m+j ,m}(\phi_{m+j+1}(c_{m+j+1,1})).
$$
This cycle contains as a subpath
\begin{equation}
\phi_{m+j ,m}(c_{m+j,1})+\eta+\phi_{m+j ,m}(c_{m+j,1})\label{eq:sc*}
\end{equation}
for each path $\eta$ of lengths $|\eta|=0,1,\ldots, k_{m+j}$.
By assumption \eqref{mix:c1} each $\phi_{m+j ,m}(c_{m+j,1})$ contains a copy of $c_{m,1}+\alpha+c_{m,1}$ with any $|\alpha|=0,\ldots, k_{m+j-1}$.

On the other hand, by assumptions \eqref{mix:c2} and \eqref{mix:c3} and \eqref{eq:sc*} there is a path $\alpha$ of length $|\alpha|\leq k_{m+j-1}$ such that
we can find in $\xi$ a path $c_{m,1}+\gamma+c_{m,1}$
with $\gamma$ of length $|\gamma|=|\alpha|+|\eta|$. 
The claim is proved, and so the result follows.
\end{proof}

As the last element we will show that $(X,T)$ does not contain an asymptotic pair. The following notion of degree was used in \cite{Shim}.
\begin{defn}
For each $v\in V_n$ we define its degree by
$$
\deg(v)=\begin{cases}
+\infty &, \text{if }v=v_{n,0},\\
i &, \text{if }v\in V(c_{n,i})\setminus \{v_{n,0}\},
\end{cases}
$$
and next for $x\in X$ we put $\deg(x)=\min_{i}\deg(x_i)$.
\end{defn}
It is not hard to see that for each $x\in X$ and each $n<m$
we have $\deg(x_n)\geq \deg(x_m)\geq 1$. Clearly, if $\deg(x)=+\infty$ then $x=p$. It is also clear that if $\deg(x)=i<+\infty$
then there is $N$ such that $x_n\in V(c_{n,i})\setminus \{v_{n,0}\}$ for all $n\geq N$.
\begin{rem}\label{rem:star}
Note that $\phi_n(c_{n+1},i)$ starts and ends with edge $e_{n,0}$
hence, if $x_n\in V(c_{n,i})\setminus \{v_{n,0}\}$ for $n\geq N$ then $T(x)_n\in V(c_{n,i})\setminus \{v_{n,0}\}$ for $n> N$.
\end{rem}

Remark~\ref{rem:star} immediately proves the following Lemma which is an analogue of Lemma~3.11 in \cite{Shim}.

\begin{lem}\label{lem:def:const}
For each $x\in X$ we have $\deg(x)=\deg(T(x))$.
\end{lem}

The following observation will simplify further considerations.
\begin{lem}\label{lem:deg:est}
Suppose that $\deg(x)=i<\infty$ and fix any $n>i$.
Then for every $m$ we have $\min_{j\geq m} \deg (T^j(x)_n)\leq i+1$.
\end{lem}
\begin{proof}
First note that if $x_j\in V(c_{j,i})\setminus \{v_{j,0}\}$ then also $x_{j+1}\in V(c_{j+1,i})\setminus \{v_{j+1,0}\}$ and so by 
Remark~\ref{rem:star}
we must have $T(x)_{j+1}\in V(c_{j+1,i})\setminus
\{v_{j+1,0}\}$. As a consequence, by induction, we easily obtain that
$T^m(x)_{j+m}\in V(c_{j+m,i})\setminus \{v_{j+m,0}\}$. But occurrence of $c_{j+m,i}$ in $\phi_{j+m}(c_{j+m+1,i})$ is eventually followed by $c_{j+m,i+1}$
and clearly $c_{j+m,i+1}$ contains a subpath $\eta$ such that $\phi_{j+m,n}(\eta)=c_{n,i+1}$.
This immediately implies that there is an integer $r\geq 0$ such that $T^{m+r}(x)_n\in V(c_{n,i+1})$.
\end{proof}

\begin{cor}\label{Cor:asymptotic}
If $x,y\in X$ are an asymptotic pair then $|\deg(x)-\deg(y)|\leq 1$. In particular there is no point $x\neq p$
asymptotic with $p$.
\end{cor}
\begin{proof}
Suppose that $x\neq y$ and they are asymptotic. Then without loss of generality we may assume that $x\neq p$
and so $\deg(x)=i<\infty$. Fix any $n>i$. Since $x,y$ are asymptotic, there exists $m$
such that $T^j(x)_n=T^j(y)_n$ for all $j\geq m$. By Lemma~\ref{lem:deg:est} we obtain that $T^j(y)_n\leq i+1$
for some $j$ and so $\deg(T^j(y))\leq i+1$, which by Lemma~\ref{lem:def:const} gives $\deg(y)\leq i+1$.
In particular $\deg(y)<+\infty$ and hence also $y\neq p$. Therefore we can change the role of $x$ and $y$ obtaining by symmetric argument that $\deg(x)\leq \deg(y)+1$.
\end{proof}

\begin{lem}\label{lem:different:deg}
If $\deg(x)=i$ and $\deg(y)=i+1$ for some $i\in \N$ then the pair $x,y$ is not asymptotic.
\end{lem}
\begin{proof}
Assume to the contrary that $\deg(y)=\deg(x)+1$, but $x$ and $y$ are asymptotic and denote $i=\deg(x)$.
For each $m,s\geq 0$ and $z\in \{x,y\}$ write $z_m^s := \phi_{\infty,m} (T^s (z))$.
Note that by Lemma~\ref{lem:def:const} we have $y_{i+1}^j \not\in V(c_{i+1,i})$ for every $j\geq 0$ and so if $x,y$ are asymptotic, then $x_{i+1}^{j}\not\in V(c_{i+1,i})$
for all $j$ sufficiently large. Therefore without loss of generality we may assume that $x_{i+1}^{j}\not\in V(c_{i+1,i})$ for all $j\geq 0$.
Fix any $n>i+1$ sufficiently large to have $\phi_{\infty,k}(x)\in V(c_{k,i})$ and $\phi_{\infty,k}(y)\in V(c_{k,i+1})$ for all $k\geq n-1$.
We also assume that $n$ is large enough to have $x_{i+1}^{j}=y_{i+1}^{j}$ for all $j\geq n$.

We can present a path induced on the graph $G_n$ by the orbit of $x$ by
$$
\seq{ x_{n}^{j} }_{j=0}^\infty =\eta + 2c_{n,i+1}+\ldots
$$
where $\eta$ is a path which does not contain complete cycle $c_{n,i+1}$ as a subpath. Clearly $\eta$ cannot contain complete cycle $c_{n,i}$ because otherwise $x_{i+1}^{j}\in V(c_{i+1,i})$ for some $j\geq 0$ which is a contradiction.
Therefore either $|\eta|\leq |c_{n,i+1}|$ or
$$\phi_{n-1}(\eta)=\eta'+2c_{n-1,i+1}+\ldots+2c_{n-1,n-1}+e_{n-1,0}$$
where $\eta'$ is a suffix of $c_{n-1,i}$, hence in this case we obtain $|\eta|\leq |\eta'|+|c_{n,i+1}|$.
Repeating the above arguments inductively, and keeping in mind that $\phi_{\infty,i+1}(x)\not\in V(c_{i+1,i})$ we obtain that
\begin{equation}
|\eta|\leq \sum_{r=1}^n |c_{r,i+1}|\leq |c_{n,i+1}|+ n |c_{n-1,i+1}|\leq |c_{n,i+1}|+|c_{n,n}|-n.\label{eta:estimate}
\end{equation}
Since $\phi_{\infty,n-1}(x)\in V(c_{n-1,i})$ we also have that $|\eta|\geq |c_{n,i+1}|$.

On the other hand, we must have $y_{i+1}^{j}\in V(c_{i+1,i+1})$ for infinitely many iterates $j$, as otherwise it is again easy to see that $x,y$ are not asymptotic.
Hence we may assume that $n$ is sufficiently large to ensure that
$$
\seq{y_{n}^{j}}_{j=0}^\infty =\xi + c_{n,i+1}+2c_{n,i+2}+\ldots + 2c_{n,n}+\ldots
$$
where $\xi$ is a suffix of $c_{n,i+1}$.
Let $j>0$ be the smallest integer such that $y_{n}^{j} \in \bigcup_{s=i+2}^nV(c_{n,s})$. Since $\phi_{\infty,n-1}(y)\in V(c_{n-1,i+1})$ we have $|\xi|> |c_{n,n}|$
and hence
$$n\leq |\eta|<j \leq  2|c_{n,i+1}|\leq |\eta|+|c_{n,i+1}|.$$
Roughly speaking, it means that $x_{n-1}^{j}$ belongs to the first occurrence of $c_{n,i+1}$ on the path $\eta+2c_{n,i+1}+\ldots$.
By the choice of $j$ we have $y_{i+1}^{j}\not\in V(c_{i+1,i+1})$ and $x_{i+1}^{j}=y_{i+1}^{j}$.
Then calculations similar to \eqref{eta:estimate} yield that there is $r< 2|c_{n,i+2}|$ such that $x_{n}^{j+r}$ is the first vertex of $c_{n,i+1}$.
This implies that $x_{i+1}^{j+r+n}\in V(c_{i+1,i+1})$ while $r+n\leq 2|c_{n,i+2}|+|c_{n,n}|$, hence we still have $y_{n}^{j+r+n}\in \bigcup_{s=i+2}^nV(c_{n,s})$
and as a consequence $y_{i+1}^{j+r+n}\not\in V(c_{i+1,i+1})$.
%
This shows that $x,y$ are not asymptotic which is a contradiction.
\end{proof}

It remains to consider the case when $x$ and $y$ have distinct finite degrees.
\begin{lem}\label{lem:overlapping}
	Let $\deg(x)=\deg(y)=i$ for some $i\in \N$ and assume that there are $n>i+2$ and $k\geq 1$ such that
	$$
	\seq{\phi_{\infty,n}(T^j(x))}_{j=k}^\infty=\seq{\phi_{\infty,n}(T^j(y))}_{j=k}^\infty=c_{n,i}+\ldots.
	$$
	Then $\seq{\phi_{\infty,n}(T^j(x))}_{j=0}^\infty=\seq{\phi_{\infty,n}(T^j(y))}_{j=0}^\infty$.
\end{lem}
\begin{proof}
For each $m,s\geq 0$ and $z\in \{x,y\}$ write $z_m^s := \phi_{\infty,m} (T^s (z))$. Then sequence $\seq{z^s_m}_{s=0}^\infty$ is an infinite path on $G_m$.
By assumption, there exist finite paths $\eta_n,\xi_n$, $|\eta_n|=|\xi_n|$ and an infinite path $\gamma_n$ on $G_n$ such that
\begin{equation}\label{eq:equal:xy}
\seq{x^s_n}_{s=0}^\infty=\eta_n+c_{n,i}+\gamma_n,\quad \seq{y^s_n}_{s=0}^\infty=\xi_n+c_{n,i}+\gamma_n.
\end{equation}
We claim that either $\eta_n=\xi_n$ or for every $m>n$ there exist finite paths $\eta_m,\xi_m$, $|\eta_m|=|\xi_m|$ such that
\begin{enumerate}
	\item $\seq{x^s_m}_{s=0}^\infty=\eta_m+c_{m,i}+\ldots$, $\seq{y^s_m}_{s=0}^\infty=\xi_m+c_{m,i}+\ldots$,
	\item $|\eta_m|+|c_{m,i}|>|\eta_{m-1}|+|c_{m-1,i}|$,
	\item\label{con:dec:m} $|\eta_m|<|\eta_{m-1}|$.
\end{enumerate}
Let us assume that $\eta_n\neq \xi_n$. This implies in particular that $x^0_m\neq y^0_m$ for every $m>n$.
Observe that the cycle $c_{n,i}$ in $\seq{x^s_n}_{s=0}^\infty$ must be contained in an image of (possibly a part of) cycle $c_{n+1,i}$ in $\seq{x^s_n}_{s=0}^\infty$
(the same is true for $\seq{y^s_n}_{s=0}^\infty$).
Picture these images written one over another as in \eqref{xandy:1}.
\begin{eqnarray}
\seq{x_{n}^s}_{s=0}^\infty &= & \overbrace{\ldots\ldots   +c_{n,i}+ \ldots}^{c_{n+1,i}} \ldots\ldots \nonumber \\
\seq{y_{n}^s}_{s=0}^\infty &=&  \ldots \underbrace{\ldots  +c_{n,i}+ \ldots\ldots}_{c_{n+1,i}}\ldots \label{xandy:1}
\end{eqnarray}
If these images are shifted, then moving to the right from $c_{n,i}$ the first vertex of $c_{n,i+1}$ must appear before trace of $\phi_{n}(c_{n+1,i})$ finish in each of these sequences, and it will appear earlier in
$\seq{y_{n}^s}_{s=0}^\infty$ than in $\seq{x_{n}^s}_{s=0}^\infty$ or vice-versa (the first case occurs if the situation is as in \eqref{xandy:1}). It leads to a contradiction with definition of $\gamma_n$.
On the other hand, whole $c_{n+1,i}$ must appear in $\seq{x_{n+1}^s}_{s=0}^\infty$ because $x_{n+1}^0 \neq y_{n+1}^0$. This proves the claim for $m=n+1$.

For the proof of the case $m>n+1$, we can repeat previous observations, however now relative shift in images of $c_{m,i}$ in paths
\begin{eqnarray}
\seq{x_{m-1}^s}_{s=0}^\infty &= & \overbrace{\ldots\ldots   +c_{m-1,i}+ \ldots}^{c_{m,i}} \ldots\ldots \nonumber \\
\seq{y_{m-1}^s}_{s=0}^\infty &=&  \ldots \underbrace{\ldots  +c_{m-1,i}+ \ldots\ldots}_{c_{m,i}}\ldots \label{xandy:2}
\end{eqnarray}
does not lead to immediate contradiction, because we cannot use $\gamma_n$ directly in this case. By the symmetry of argument, we may assume that the relative shift is as in
\eqref{xandy:2}. As before,  moving to the right from $c_{m-1,i}$ in both paths on $G_{m-1}$, the first occurrence of a vertex of $c_{m-1,i+1}$ will appear earlier in
$\seq{x_{m-1}^s}_{s=0}^\infty$ than in $\seq{y_{m-1}^s}_{s=0}^\infty$. Note that trace of $\phi_{m-1,n}(c_{m-1,i+1})$ in $\seq{x_n^s}_{s=0}^\infty$ contains a copy of cycle $c_{n,i+1}$ in $\gamma_n$, 
hence $c_{n,i+1}$ will also appear in $\seq{y_{n}^s}_{s=0}^\infty$. But we know that it cannot be obtained as image of $c_{m-1,i+1}$ (it is a too early position to see a vertex of $c_{m-1,i+1}$), hence the only possibility is that it is a part of image $\phi_{m-1,n}(c_{m-1,i})$.
But then at level $m-2$ we will see one of the following two paths:
\begin{eqnarray}
\seq{y_{m-2}^s}_{s=0}^\infty &= & \overbrace{\ldots  \ldots+c_{m-2,m-2}+e_{m-2,0}+c_{m-2,i+1}+\ldots}^{c_{m,i}} \ldots\ldots \nonumber \\
\seq{x_{m-2}^s}_{s=0}^\infty &=&  \ldots \underbrace{.\,\ldots \ldots\ldots\ldots+c_{m-2,i}+c_{m-2,i+1}+\ldots\ldots}_{c_{m,i}}\ldots \label{xandy:3}
\end{eqnarray}
or
\begin{eqnarray}
\seq{x_{m-2}^s}_{s=0}^\infty &= & \overbrace{\ldots\ldots   +c_{m-2,i+1}+c_{m-2,i+1}+ \ldots}^{c_{m,i}} \ldots\ldots. \nonumber \\
\seq{y_{m-2}^s}_{s=0}^\infty &=&  \ldots \underbrace{\ldots  +c_{m-2,i+1}+c_{m-2,i+2}+ \ldots\ldots}_{c_{m,i}}.\ldots \label{xandy:4}
\end{eqnarray}
Note that in the case \eqref{xandy:3} if we move left from $c_{m-2,i+1}$ then in $\seq{y_{m-2}^s}_{s=0}^\infty$ we will reach the first position $r$ (within $c_{m-2,i}$)
which $\phi_{m-2,n}(x^r_{m-2})$ pointing at the last vertex of $c_{n,i}$ while $x_n^r\not\in V(c_{n,i})$. `But the position $r$ occurs after the start of $\gamma_n$ which is a contradiction.

In the case \eqref{xandy:4}, now moving to the right from $c_{m-2,i+1}$ in $\seq{x_{m-2}^s}_{s=0}^\infty$
will eventually reach the first position $r$ (within $c_{m-2,i+1}$)
which $\phi_{m-2,n}$ sends to the first vertex of $c_{n,i+1}$, while still $y_n^r\in V(c_{n,i+1})$. A contradiction again. The claim is proved.

To complete the proof, note that the condition \eqref{con:dec:m} of the claim cannot be satisfied for infinitely many $m$, hence the only possibility is that $\eta_n=\xi_n$
and consequently $\seq{x_{n}^s}_{s=0}^\infty=\seq{y_{n}^s}_{s=0}^\infty$.
\end{proof}

\begin{lem}\label{lem:deg:xeqy}
Let $\deg(x)=\deg(y)=i$ for some $i\in \N$. If the pair $x,y\in X$ is asymptotic, then  $x=y$.
\end{lem}

\begin{proof}
Let us choose an asymptotic pair $x,y\in X$, $x\neq y$ for which  $\deg(x)=\deg(y)=i$ for some $i\in \N$.
There exists an integer $N> i+2$, for which $\deg (x_n) = \deg (y_n) = i$ for all $n\geq N$. We may also assume that $x_n\neq y_n$ for all $n\geq N$.	
As before, for each $m,s\geq 0$ and $z\in \{x,y\}$ write $z_m^s := \phi_{\infty,m} (T^s (z))$.
Since the pair $x,y$ is asymptotic, there exists $K>0$ such that
$x_N^k = y_N^k$ for all $k\geq K$.
If the sequence $\seq{x_N^k}_{k=0}^\infty$ contains infinitely many copies of $c_{N,i}$ then we may apply Lemma~\ref{lem:overlapping}
obtaining that $x_N^k = y_N^k$ for every $k\geq 0$, in particular $x_N=y_N$, which is a contradiction.
Hence the only possibility is that $\seq{x_N^k}_{k=0}^\infty$ contains finitely many copies of $c_{N,i}$. Note that each complete cycle $c_{n,i}$ appering in $\seq{x_n^k}_{k=0}^\infty$ induces at least two copies of cycle $c_{n-1,i}$
in $\seq{x_{n-1}^k}_{k=0}^\infty$, hence without loss of generality we may assume that for all $n\geq N$, complete cycle $c_{n,i}$ does not appear neither in $\seq{x_n^k}_{k=0}^\infty$ nor in $\seq{y_n^k}_{k=0}^\infty$.
Note that $x_n, y_n\in c_{n,i}$ for all $n\geq N$ and for each $n$ there exist suffixes $\eta_n$ and $\xi_n$ of $c_{n,i}$, starting with vertex $x_n$ and $y_n$ respectively, which is a prefix of path $\seq{x_n^k}_{k=0}^\infty$ and $\seq{y_n^k}_{k=0}^\infty$. First observe that by the definition of $c_{n,i}$ and $\phi_n$ we must have
$$
\lim_{n\to\infty}|\eta_n|=\lim_{n\to\infty}|\xi_n|=+\infty.
$$
We may also assume that $|\eta_n|<|\xi_n|$ for infinitely many $n$. If $\limsup_{n\to\infty}||\eta_n|-|\xi_n||=+\infty$ then  for each $M$ there are $n>N$ and
$j>M$ such that $x_N^j=\phi_{n,N}(x_n^j)=y_N\in V(c_{N,i})$ or $y_N^j=\phi_{n,N}(y_n^j)=x_N\in V(c_{N,i})$ which is a contradiction.
We obtain that the sequence $\seq{||\eta_n|-|\xi_n||}_{n=N}^\infty$ is bounded, and therefore there exists an integer $m\in \Z$ such that $T^m (x) = y$.
If $m=0$ then we are done, so without loss of generality assume that $m\neq 0$, which implies that $x$ is a periodic point.
But $\deg(x)<\infty$, so $x\neq p$ contradicts Lemma~\ref{lem:prox}. The proof is completed.
\end{proof}

\section{Applications}
\subsection{Proof of Theorem~\ref{main}}

Consider the dynamical system $(X,T)$(=$(V_\G,T_\G)$) constructed in Section~\ref{sec:constr} and denote by $p$ its unique (fixed) point with $\deg(p)=+\infty$. By Corollary~\ref{cor:prox} $(X,T)$ is proximal, and by Lemma~\ref{lem:mixing}
it is topologically mixing. Take any $x,y\in X$ which are asymptotic. By Corollary~\ref{Cor:asymptotic} and Lemma~\ref{lem:different:deg} we have either $x=y=p$ or there is $i\in \N$
such that $\deg(x)=\deg(y)=i$. By Lemma~\ref{lem:deg:xeqy} we obtain that also in the second case $x=y$.
This shows that when $x\neq y$ then they are proximal but not asymptotic. Indeed $(X,T)$ is completely scrambled.

\subsection{Solution to Question~\ref{que:1}}\label{sec:que}

Let $(X,T)$ be a transitive completely scrambled homeomorphism acting on a compact metric space $X$ of topological dimension $n\geq 0$.
We will show how to extend it to a system $(Y,S)$ with connected $Y$ of topological dimension $n+1$. Note that the case of $(X,T)$ for $n=0$ is provided by Theorem~\ref{main}, hence by induction using the method described
below, we can construct it for any $n>0$. This will solve Question~\ref{que:1}.

Let $\seq{\Phi_t}_{t\in \R}$ be the suspension flow of $(X,T)$ acting on a corresponding space $Z$ obtained as a respective quotient of $X\times [0,1]$. Since $(X,T)$ is transitive, it is not hard to verify that
there is an irrational $s>0$ such that the homeomorphism $R=\Phi_s\colon Z\to Z$ is transitive (in fact there is a residual set of parameters $t\in \R$ for which $\Phi_t$ is transitive; e.g. see the first part of the proof of \cite[Proposition~4]{Fayad}).
Since there is no asymptotic pair in $(X,T)$ it is also clear that $(Z,R)$ does not contain asymptotic pairs. Let $p$ be a fixed point which is the unique minimal set for $(X.T)$. Then it is not hard to see that the set $M=\omega_R(p,0)$
is a minimal set for $(Z,R)$, homeomorphic to the unit circle. Now fix any $(x,r)\in Z$. Since $(p,x)$ are proximal in $(X,T)$, for every $m>0$ and open neighborhood $V$ of $p$ there is $i>0$ such that $T^{i+j}(x)\in V$
for all $j=0,\ldots,m$. But $m$ can be arbitrarily large, so for every neighborhood $U$ of $(p,0)$ there exists $n>0$ such that $R^n(x,q)\in U$, so in particular $(p,0)\in \omega_R(x,r)$.
This proves that $M$ is the unique minimal set for $(Z,R)$. Let $Y$ be obtained from $Z$ by collapsing $M$ to a point. This induces in a natural way a factor $(Y,S)$ with a factor map $\pi\colon (Z,R)\to (Y,S)$.
Denote by $q$ the unique fixed point of $(Y,S)$ defined by $\{q\}=\pi(M)$. It is obvious that there is no other minimal subset in $(Y,S)$, hence $(Y,S)$ is proximal (e.g. see \cite{Akin}).
Additionally observe that there is no nontrivial asymptotic pair in $(Y,S)$. Since there is no asymptotic pair in $(Z,R)$ consisting of distinct points, the only possibility could be that there is $y\in Y$, $y\neq q$ asymptotic to $q$.
But then there exists $(z,r)\in Z$ such that $\omega_R(z,t)=\omega_R(p,0)$, which in turn implies that $\omega_T(z)=p$ which is a contradiction.

We have just proved that $(Y,S)$ is proximal without nontrivial asymptotic pairs, which equivalently means that $(X,S)$ is completely scrambled. It follows directly from the construction
that $Y$ is of topological dimension $n+1$. We have proved the following.
\begin{thm}
For every $n\geq 1$ there exists a continuum $X$ of dimension $n$ supporting a completely scrambled transitive homeomorphims.
\end{thm}

\subsection{Proof of Theorem~\ref{main:2}}
First fix any $n\geq 2$.
Apply Theorem~\ref{main} to obtain a mixing completely scrambled homeomorphism $(X,T)$ where $X$ is a Cantor set (strictly speaking, it is the dynamical system $(V_\G,T_\G)$ constructed in Section~\ref{sec:constr}).
On the other hand, Glasner and Maon, inspired by an earlier work of Glasner and Weiss \cite{GW}, provided in 1989 a method of construction of minimal, weakly mixing and uniformly rigid homeomorphisms \cite[Proposition~6.5]{Maon}.
As a particular example of application of this techniques we obtain a uniformly rigid, weakly mixing and minimal dynamical system $(Z,R)$ where $Z=\mathbb{T}^n$ is $n$-dimensional torus.
Denote by $p$ the unique fixed point of $(X,T)$.
Let $(Y,S)$ be a factor of $(X\times Z, T\times R)$ obtained by collapsing minimal set $M=\{p\}\times Z$ to a point $q\in Y$, and let $\pi \colon X\times Z, T\times R) \to (Y,S)$ be a corresponding factor map.
Assume that $x,y\in Y$ are distinct points forming an asymptotic pair and take any $u,v\in Z$, $s,t\in X$ such that $\pi(s,u)=x$ and $\pi(t,v)=y$. If $s\neq p$ then there is an open set $V\supset M$
such that $(T\times R)^n(s,u)\not\in V$ for infinitely many $n\geq 0$. Let us fix one such $n$ and any open set $U\supset M$ such that $\overline{U}\subset V$. Take any $\lambda>0$ such that $\dist (X\setminus V,\overline{U})>0$
and $\limsup_{i\to\infty} d(T^i(s),T^i(t))>\lambda$. Since we may assume that $\pi$ outside $U$ does not decrease distances, it is not hard to see that $d(T^n(x),T^n(y))>\lambda/2$ which is impossible.
But then the only remaining possibility is that $s=t=p$ which is also impossible, because $x\neq y$. This shows that distinct points in $Y$ are never asymptotic. But it is also clear that $\{q\}$ is the unique minimal subset in $Y$
which shows that $(Y,S)$ is proximal. Finally, $(Y,S)$ is weakly mixing as a factor of a weakly mixing dynamical system $(X\times Z, T\times R)$. We have just proved Theorem~\ref{main:2} for $n\geq 2$.

Finally, to prove Theorem~\ref{main:2} for $n=1$  it suffices to use Handel's minimal homeomorphism on the pseudo-circle \cite{Handel}, since this homeomorphism is weakly mixing and uniformly rigid \cite{BCO}.


\section{Acknowledgments}
The third named author is grateful to Jian Li, Wen Huang and Takashi Shimomura for many fruitful discussions on topics related to this paper.

J. Boro\'nski's work was supported by National Science Centre, Poland (NCN), grant no. 2015/19/D/ST1/01184. Research of P. Oprocha was supported by National Science Centre, Poland (NCN), grant no. 2015/17/B/ST1/01259.

\end{document}